\documentclass[12pt]{amsart}
\usepackage{amsmath}
\usepackage[all]{xy}
\usepackage[boxsize=1.25em, centerboxes]{ytableau}

\usepackage[hidelinks]{hyperref}

\usepackage{amsthm,amsfonts, amssymb, amscd}
\usepackage{youngtab}
\usepackage{ytableau}
\usepackage[all]{xy}
\usepackage{multirow, multicol}
\usepackage{euscript}
\usepackage{cite}

\usepackage{tikz}
\usetikzlibrary{matrix,arrows, decorations.markings}
\usetikzlibrary{positioning}

\usepackage[normalem]{ulem}

\hypersetup{
    colorlinks = true, 
    linkbordercolor = {blue},
    citecolor = {blue}
}

\let\oldtocsection=\tocsection
\let\oldtocsubsection=\tocsubsection
\renewcommand{\tocsection}[2]{\hspace{0em}\oldtocsection{#1}{#2}}
\renewcommand{\tocsubsection}[2]{\hspace{1em}\oldtocsubsection{#1}{#2}}

\newtheorem{thm}{Theorem}[section]
\newtheorem{lemma}[thm]{Lemma}
\newtheorem{cor}[thm]{Corollary}
\newtheorem{prop}{Proposition}[section]

\theoremstyle{remark}
\newtheorem{rem}[thm]{Remark}
\theoremstyle{remark}
\newtheorem{example}[thm]{Example}
\theoremstyle{definition}

\newtheorem{THM}{Theorem}
\theoremstyle{definition}

\theoremstyle{definition}

\theoremstyle{definition}

\newtheorem{definition}[thm]{Definition}

\numberwithin{equation}{section}

\newcommand{\C}{\mathbb{C}}           
\newcommand{\Z}{\mathbb{ Z}}           
 
\newcommand{\rank}{\operatorname{rank}}

\newcommand{\diag}{\operatorname{diag}}

\newcommand{\rk}{\operatorname{rk}}

\newcommand{\GL}{\operatorname{GL}}

\newcommand{\cc}{\mathcal{C}}

 \newcommand{\ci}{\mathcal{I}}
 \newcommand{\cj}{\mathcal{J}}
 \newcommand{\ck}{\mathcal{K}}
 \newcommand{\cl}{\mathcal{L}}

 \newcommand{\cx}{\mathcal{X}}

\renewcommand{\tilde}{\widetilde}

\newcommand{\cy}{\mathcal{Y}}

\newcommand{\op}{\mathrm{op}}
\newcommand{\Spec}{\mathrm{Spec}}
\newcommand{\rad}{\mathrm{rad}}
\newcommand{\Br}{\mathrm{Br}}

\newcommand{\semi}{\mathsf{s}}
\newcommand{\nilp}{\mathsf{n}}

\newcommand{\mx}{\mathsf{x}}

\begin{document}

\title[Minimal Semisimple Hessenberg Schemes]{Minimal Semisimple Hessenberg Schemes}

\author{Rebecca Goldin}
\address{Department of Mathematical Sciences\\ George Mason University\\ 4400 University Drive\\ Fairfax, VA\\ 22030\\ USA}
\email{rgoldin@gmu.edu }

\author{Martha Precup}
\address{Department of Mathematics\\ Washington University in St. Louis \\ One Brookings Drive\\ St. Louis, Missouri\\ 63130\\ USA  }
\email{martha.precup@wustl.edu}

\date{\today}

\begin{abstract} We study a collection of Hessenberg varieties in the type A flag variety associated to a nonzero semisimple matrix whose conjugacy class has minimal dimension.  We prove each such minimal semisimple Hessenberg variety is a union Richardson varieties and compute this set of Richardson varieties explicitly. Our methods leverage the notion of matrix Hessenberg schemes to answer questions about the geometry of minimal semisimple Hessenberg varieties using commutative algebra and known results on Schubert determinantal ideals. In particular, we show that all type A minimal semisimple matrix Hessenberg schemes are reduced.
\end{abstract}

\maketitle


\section{Introduction}
Let $G:=GL_n(\C)$ be the set of $n\times n$ invertible matrices and $B$ be the subgroup of upper triangular matrices in $G$. A Hessenberg variety $Y_{\mx, h}$ is a subvariety of the flag variety $G/B$ defined by a choice of matrix $\mx\in \mathfrak{gl}_n(\C) = \mathrm{Lie}(G)$ and Hessenberg function~$h$.
In this paper, we consider Hessenberg varieties in the flag variety $G/B$ for $\mx$ a \emph{minimal semisimple element}, that is, $\mx$ is a non-scalar diagonalizable matrix such that the dimension of its conjugacy class in $\mathfrak{gl}_n(\C)$ is as small as possible. In general, we say that any nonzero matrix $\mx$ is \emph{minimal} if its conjugacy class satisfies this property. If $\mx$ is a minimal element, then we say that $Y_{\mx, h}$ is a \textit{minimal Hessenberg variety}.   

Consider the minimal nilpotent element $\nilp=E_{1n}\in \mathfrak{gl}_n(\C)$, where $E_{1n}$ is the elementary matrix with unique nonzero entry equal to $1$ in row $1$ and column $n$.  Conjugation by $B$ scales $\nilp$ so $Y_{\nilp ,h}$ is $B$-invariant, implying that $Y_{\nilp ,h}$ is a union of Schubert varieties.
The minimal nilpotent Hessenberg varieties $Y_{\nilp, h}$ were studied by Tymoczko \cite{Tymoczko2006A} and by Abe and Crooks \cite{Abe-Crooks2016}, where they identify 
components of $Y_{\nilp, h}$ explicitly. Here we characterize $Y_{\mx, h}$ in the minimal semisimple case, and show that they are unions of Richardson varieties determined by combinatorial data associated with the Hessenberg function $h$. We 
describe their defining ideals, allowing us to prove they are reduced as schemes. In a companion paper~\cite{GP-families} we prove that these Hessenberg schemes lie in flat families over lines in the union of all minimal orbits, and determine the (generally non-reduced) scheme structure on minimal nilpotent Hessenberg varieties as well.

Throughout this paper, we identify $G/B$ with the collection of complete flags in $\C^n$,
$$
V_\bullet = 0\subset V_1\subset V_2\subset \cdots \subset V_{n-1}\subset V_n=\C^n
$$
with $\dim_\C V_i =i$ for all $1\leq i\leq n$.  Denote by $B_-$ the subgroup of lower triangular matrices. Then $T:=B\cap B_-$ is a maximal torus, consisting of diagonal matrices in $G$. We identify the Weyl group of $G$ with the permutation group $S_n$ on $n$ letters. Given $w\in S_n$ we let $\dot w$ denote the permutation matrix with entry equal to $1$ in row $i$ and column $j$ if $w(j)=i$.  Our convention is that the \emph{Schubert variety} $X_w=\overline{B \dot wB/B}$ is the closure of the left $B$-orbit of $\dot wB$ in $G/B$, while the \emph{opposite Schubert variety} $X_w^{\op} = \overline{B_-\dot wB/B}$ is the closure the $B_-$-orbit of $\dot wB$.

Hessenberg varieties are parameterized by a nondecreasing function 
$$
h: \{1,2, \dots, n\}\rightarrow \{1,2,\dots, n\}
$$ 
satisfying $h(i)\geq i$, called a \emph{Hessenberg function}, and an $n\times n$ matrix 
$\mx\in \mathfrak{gl}_n(\C)$.  Explicitly,
\begin{eqnarray}\label{eqn.hessdef}
Y_{\mx, h} := \{V_\bullet \mid \mx(V_i)\subseteq V_{h(i)} \}
\end{eqnarray}
is the \emph{Hessenberg variety} corresponding to the matrix $\mx$ and Hessenberg function $h$.
If $\mx$ is a nilpotent (respectively, semisimple) element of $\mathfrak{gl}_n(\C)$ then we say that $Y_{\mx, h}$ is a \emph{nilpotent} (respectively, \emph{semisimple}) \emph{Hessenberg variety}.

We employ the term \emph{variety} for $Y_{\mx, h}$ even though its defining ideal may not be reduced. The literature uses the term ``Hessenberg variety" both to refer to the underlying reduced variety and to the scheme corresponding to the defining ideal of~\eqref{eqn.hessdef}. Our results include a description of $Y_{\mx,h}$ as a scheme; we make this explicit in Section~\ref{sec.Hess.schemes}, where we lift $Y_{\mx,h}$ to a subscheme of $M_n(\C)$, the affine scheme of $n\times n$ complex matrices.

In this article we characterize minimal semisimple Hessenberg varieties $Y_{\mx, h}$.  We carry out this program by working ``upstairs" in the matrix group $G$, realized as an open subscheme of $M_n(\C)$.
Following Insko--Tymoczko--Woo~\cite{Insko-Tymoczko-Woo} we introduce the \emph{matrix Hessenberg scheme},  a right $B$-invariant subscheme $\cy_{\mx, h}$ of $M_n(\C)$ equal to the closure of $\pi^{-1}(Y_{\mx, h})\subseteq G$, where $\pi: G\rightarrow G/B$ is the quotient map.

We show in Sections~\ref{sec.Hess.schemes} and~\ref{sec.proofs} that, for all $h$, it suffices to consider  the matrix Hessenberg scheme defined using the minimal semisimple element
\[
\semi:= \mathrm{diag}(1,0,\ldots, 0) \in \mathfrak{gl}_n(\C).
\]
In Theorem~\ref{thm.ss.components}, we 
identify all components of $\cy_{\semi, h}$ and prove that it is reduced for all Hessenberg functions $h$. We prove that the irreducible components of $\cy_{\semi, h}$ are in bijection with the \emph{corners} $\cc(h)$ of $h$, the set of $i\in [n]$ such that $h(i)>h(i-1)$ where, by convention, $h(0)=0$. 
As a consequence,
we obtain a corresponding 
decomposition of the Hessenberg variety $Y_{\semi, h}$ into a reduced union of Richardson varieties.

\begin{THM}\label{thm1} For all Hessenberg functions $h$, the minimal semisimple Hessenberg variety $Y_{\semi, h}$ is equal to a reduced union of Richardson varieties,
\[
Y_{\semi, h} = \bigcup_{i\in \cc(h)} \left( X_{u[i]}^{\op} \cap X_{w_0 v[h(i)]} \right)
\]
where $u[i]\in S_n$ is shortest permutation $u$ such that $u(i)=1$ and $v[j]\in S_n$ is the shortest permutation $v$ such that $v(j) = n$. 
\end{THM}

While we show the minimal semisimple Hessenberg variety $Y_{\semi, h}$ is reduced for all $h$, it may not be equidimensional 
(see Corollary~\ref{cor.equidim}). 
The top-dimensional components are detected by its cohomology class in $H^*(G/B)$, which we describe as a function of $h$ in Corollary~\ref{cor.cohomclass}.
In a companion paper, we show that the minimal semisimple and nilpotent matrix Hessenberg schemes lie in a flat family \cite{GP-families}. This provides a unified explanation for surprising similarities between minimal semisimple and minimal nilpotent Hessenberg varieties.

\subsection{Acknowledgments} 
The first author was supported by National Science Foundation grant \#2152312 and the second author was supported by grant \#1954001 and \#2237057. We thank Rahul Singh for several helpful conversations.


\section{Background}\label{sec.background}

This section contains the definitions and basic properties of Schubert determinantal ideals that will be used below. Many of the results in this section were originally proved by Fulton~\cite{Fulton} and also Knutson and Miller~\cite{Knutson-Miller2005}; Chapter 15 of the text~\cite{MS05} is an excellent expository reference. 

Let $n$ be a positive integer and $[n]:=\{1,2,\ldots, n\}$. As in the introduction, $G:=GL_n(\C)$ is the algebraic group of $n\times n$ invertible complex matrices with Lie algebra $\mathfrak{gl}_n(\C)$.  For any $n\times n$ matrix $A$, let $A_{p\times q}$ denote the submatrix of the top $p$ rows and left-most $q$ columns, i.e.~the upper $p\times q$ northwest submatrix of $A$. 


Let $w\in S_n$.
For any $p,q$ with $1\leq p,q\leq n$, the rank of $\dot{w}_{p\times q}$ is given by
\[
r_{pq}(w) := \rank(\dot{w}_{p\times q})= \left| \{w(1), \ldots, w(q)\} \cap \{1, \ldots, p\} \right|.
\]  
Let $\mathbf{z}:= \{z_{ij} \mid (i,j) \in [n]\times [n]\}$ denote $n^2$ variables, and  $Z = (z_{ij})_{1\leq i, j \leq n}\in M_n(\C[\mathbf{z}])$ be the matrix with variable $z_{ij}$ in the $ij$th entry of the matrix.
The \emph{Schubert determinantal ideal $\ci_w$ corresponding to the permutation $w$} is the ideal in $\C[\mathbf{z}]$ generated by all minors of $Z_{p\times q}$ of size $1+r_{pq}(w)$.\footnote{ \cite{MS05} and \cite{Knutson-Miller2005} use the transpose matrix $\dot{w}^{T}$ rather than $\dot{w}$, resulting in a different ideal $\cj_w$. Transposition on $M_n(\C)$  induces an isomorphism on the coordinate ring that sends $\cj_w$ to $\ci_w$. Geometrically, transposition descends to an isomorphism from $B_-\backslash G$ to $G/B$.}

For each $w\in S_n$,  the \emph{matrix Schubert variety} is given by $\cx_w:= \overline{B\dot wB}$ in $M_n(\C) = \Spec(\C[\mathbf{z}])$.  Similarly, the \emph{opposite matrix Schubert variety} is $\cx_w^{\op}:= \overline{B_-\dot wB}$.\footnote{The varieties $\cx_w^{\op}$ are called `matrix Schubert varieties'  
by some authors.}
We will frequently use the relationship that
\begin{equation}\label{eq:oppositetoregularSchubert}
w_0 \cx_w^{\op}  = \cx_{w_0w},
\end{equation}
where $w_0$ is the longest word in $S_n$. The following is well-known, see~\cite[Chap.~15]{MS05}.

\begin{prop}\label{prop.Schubert} Let $w\in S_n$ and $\ci_w$ denote the corresponding Schubert determinantal ideal. Recall that $\ell(w):= |\{(i,j) \mid 1\leq i<j\leq n, w(i)>w(j) \}|$ is the \emph{length} of $w$.
\begin{enumerate}
\item The opposite matrix Schubert variety $\cx_w^{\op}$ is the reduced subscheme $\Spec (\C[\mathbf{z}]/\ci_w)$; in particular, Schubert determinantal ideals are prime. 
\item  Let $\pi: G\rightarrow G/B$ be the natural projection. Then $\cx_w^{\op} \cap G = \pi^{-1}(X^{\op}_w)$.
\item  The dimension of the opposite Schubert variety $\cx_w^{\op}$ is
$\dim \cx_w^{\op} = n^2 - \ell(w)$.
\end{enumerate}
\end{prop}
By Equation~\eqref{eq:oppositetoregularSchubert} and statement (1) above, we find that 
\begin{eqnarray}\label{eqn.Schubertscheme}
\cx_{w_0w}=\Spec(\C[\mathbf{z}]/  w_0\cdot \ci_w)
\end{eqnarray}
where $w_0\cdot \ci_w$ is the image of the Schubert determinantal ideal $\ci_w$ under the action of $w_0$ on $\C[\mathbf{z}]$ defined on variables $z_{ij}$ by $w_0\cdot z_{ij} = z_{w_0(i)j}$.

Inclusions of matrix Schubert varieties (and also Schubert varieties) define a partial order $\leq_{\Br}$ on the set of permutations known as \emph{Bruhat order}.  That is, we have $v\leq_\Br w$ if and only if $\cx_v \subseteq \cx_w$. Consequently, $v\leq_\Br w$ if and only if $\cx^\op_{v}\supseteq \cx^\op_{w}$.

Particular Schubert and opposite Schubert varieties play a role in characterizing minimal semisimple Hessenberg schemes.  
For each $k\in [n]$, let
\begin{itemize}
\item $u[k]$ be the shortest permutation $u$ in $S_n$ such that $u(k)=1$ and
\item $v[k]$ be the shortest permutation $v$ in $S_n$ such that $v(k)=n$.
\end{itemize}
It is easy in practice to write down the one-line notation for each such permutation.  Indeed, we obtain the one-line notation for $u[k]$ by placing $1$ in the $k$-th position and filling all remaining entries with $2, \ldots, n$ in increasing order from left to right.  Similarly, we obtain the one-line notation for $v[k]$ by placing $n$ in the $k$-th position and filling all remaining entries with $1, \ldots, n-1$ in increasing order from left to right.  For example, when $n=5$, $u[2] = [2, 1, 3, 4, 5]$ and $v[3] = [1, 2, 5, 3, 4]$. Note that $u[1] = v[n] = e$ is the identity.

We now introduce the particular ideals of interest in this paper. For each $i\in [n-1]$, set
\begin{equation}\label{eq:LiKi}
\begin{aligned}
\cl_i&:= \left< z_{11}, z_{12}, \ldots, z_{1i} \right>, \quad \mbox{and}\\
\ck_i &:= \left<p_B \mid B\subseteq \{2, \ldots n\}, |B|=i \right>
\end{aligned}
\end{equation}
where $p_B$ is the $i\times i$ minor of  $Z = (z_{ij})_{1\leq i, j \leq n}$ given by $\det(z_{ k\ell})_{k\in B, 1\leq \ell\leq i}$.
Define $\mathcal L_0 := \left< 0 \right>$ and $\ck_n:=\left< 0 \right>$.

The following  lemma is straightforward to verify using the definition of a matrix Schubert variety and the observation of Equation~\eqref{eqn.Schubertscheme}.

\begin{lemma}\label{lem.Schubert} For all $i\in [n-1]$, $\cl_i$ is the Schubert determinantal ideal for the permutation $u[i+1]$ and $w_0\cdot\ck_i$ is the Schubert determinantal ideal for the permutation $v[i]$. Explicitly,
\[
\cx_{u[i+1]}^{\op} = \Spec(\C[\mathbf{z}]/\cl_i) \quad\mbox{and} \quad\cx_{w_0v[i]} = \Spec(\C[\mathbf{z}]/\ck_i).
\]
\end{lemma}

Subvarieties of the flag variety equal to the intersection of a Schubert and opposite Schubert variety are called \emph{Richardson varieties}.  Let $u,v\in S_n$ and recall that the Richardson variety $X_u^{\op}\cap X_v$ is nonempty if and only if $u\leq v$ in Bruhat order.  In that case, we say that the closure 
$$
\overline{\pi^{-1}(X_u^{\op}\cap X_v)} \subseteq M_n(\C)
$$
 is a \emph{matrix Richardson variety}. The next lemma shows that the sum of two of the  Schubert ideals considered above defines a matrix Richardson variety.

\begin{lemma} \label{lem.Schubert4} For all $i,j \in [n]$, the ideal $\cl_i + \ck_{j}$ is prime.  It follow that, for all $1\leq i < j \leq n$, the intersection 
\[
\cx_{u[i+1]}^\op \cap \cx_{w_0v[j]} = \Spec(\C[\mathbf{z}]/ (\cl_i + \ck_{j}))
\]
is a matrix Richardson variety. 
\end{lemma}
\begin{proof} 
Observe that $\cl_i = \left< z_{11}, z_{12}, \ldots, z_{1i} \right>$ and $\ck_j \subseteq \C[\mathbf{z}']$ with $\mathbf{z'} = \{z_{ij} \mid 2\leq i \leq n, 1\leq j \leq n\}$ are prime ideals  whose generators have no variables in common. Thus the sum $\cl_i + \ck_j$ is prime, and the affine subscheme $\cx_{u[i+1]}^{\op} \cap \cx_{w_0v[j]}$ in $\C[\mathbf{z}]$ defined by $\cl_i + \ck_{j}$ is reduced. One may check that $u[i+1]\leq_{\Br} w_0v[j]$ if and only if $i<j$, implying $X_{u[i+1]}^{\op} \cap X_{w_0v[j]}$ nonempty exactly when $i<j$.  Finally observe that $\pi^{-1}(X_{u[i+1]}^{\op} \cap X_{w_0v[j]})= \pi^{-1}(X_{u[i+1]}^{\op} )\cap \pi^{-1}(X_{w_0v[j]})\subseteq \cx^{op}_{u[i+1]}\cap \cx_{w_0v[j]}$, so the intersection being reduced implies 
$$
\overline{\pi^{-1}(X_{u[i+1]}^{\op} \cap X_{w_0v[j]})}=\cx_{u[i+1]}^{\op} \cap \cx_{w_0v[j]}.
$$ 
\end{proof}

 Toward the proof of Theorem~\ref{thm1}, we highlight several needed properties of the ideals $\cl_i$ and $\ck_{j}$ in the following lemmas.

\begin{lemma}\label{lemma.inclusions} The following statements all hold for $i,j,k,\ell\in [n-1]$.
\begin{enumerate}
\item Suppose $i\leq j$. Then $\cl_i \subseteq \cl_j$ and $\ck_{j}\subseteq \ck_i$.
\item Suppose $i<j$ and $k<\ell$.  Then $\cl_i+ \ck_j \subseteq \cl_k+ \ck_\ell$ if and only if $i\leq k$ and $j \geq \ell$.
\end{enumerate}
\end{lemma}
\begin{proof}  The assertions of (1) follow easily from the definitions.

To prove (2), observe that $i\leq k$ and $j\geq \ell$ immediately implies $\cl_i+ \ck_j \subseteq \cl_k+ \ck_\ell$ by (1). Suppose now that $\cl_i+ \ck_j \subseteq \cl_k+ \ck_\ell$.  By Lemma~\ref{lem.Schubert4}, there an inclusion of Richardson varieties,
\[
X_{u[k+1]}^{\op} \cap X_{w_0v[\ell]} \subseteq X_{u[i+1]}^{\op} \cap X_{w_0v[j]}.
\]
The inclusion implies $u[i+1]\leq_{\Br} u[k+1]$, which in turn implies $i\leq k$.
Similarly, the inclusion implies $w_0v[\ell]\leq_{\Br} w_0v[j]$, or, equivalently, $v[\ell]\geq_{\Br} v[j]$. A quick check shows $j\geq \ell$.
\end{proof}

Next, we argue that the product of ideals $\cl_i \cdot \ck_j$ is radical.  

\begin{lemma}\label{lem.intersection} For all $i,j\in [n-1]$, the product $\cl_i\cdot \ck_j$ is radical. In particular, $\cl_i\cdot \ck_j = \cl_i \cap \ck_j$. 
\end{lemma}
\begin{proof} Let $\preceq$ be a monomial order on $\C[\mathbf{z}]$ and recall that if the monomial ideal generated by leading terms of polynomials in an ideal $\ci$ with respect to $\preceq$ is square free, then $\ci$ is radical \cite[Prop.~3.3]{Herzog-Hibi}. This fact can be used to prove that $\cl_i\cdot \ck_j$ is radical using a simple Gr\"obner basis argument (see~\cite{CLO} for background on Gr\"obner bases); we provide an outline below for the interested reader.

Let $\preceq$ be any diagonal term order on $\C[\mathbf{z}]$. It is well known (see e.g.~\cite[Thm.~B]{Knutson-Miller2005}) that
\[
\{ p_B \mid B\subseteq \{2, \ldots, n\}, |B|=j\}
\]
is a Gr\"obner basis for $\ck_j$ with respect to this ordering and it's obvious that $\{z_{11}, z_{12}, \ldots, z_{1i}\}$ is a Gr\"obner basis for $\cl_i$. Since the leading monomial of any $p_B$ and any $z_{1k}$ are relatively prime, it is straightforward to check that 
\[
\{ z_{1k}p_B \mid  B\subseteq \{2, \ldots, n\}, |B|=j, 1\leq k \leq i \}
\]
is a Gr\"obner basis for $ \cl_i \cdot \ck_j$. The leading term of every polynomial in this Gr\"obner basis is square free, so $\cl_i \cdot \ck_j$ is radical.  Thus
\begin{eqnarray*}
\cl_i \cdot \ck_j = \rad(\cl_i \cdot \ck_j) =  \rad(\cl_i \cap \ck_j ) = \rad(\cl_i )\cap \rad(\ck_j) = \cl_i \cap \ck_j
\end{eqnarray*}
as desired.
\end{proof}


\section{Matrix Hessenberg schemes} \label{sec.Hess.schemes}

Recall from the Introduction that a function $h:[n]\to [n]$ is a \emph{Hessenberg function} if $h(i-1)\leq h(i)$ and $h(i)\geq i$ for all $i$.
We denote a Hessenberg function $h:[n]\to [n]$  by listing its values, e.g.~$h=(h(1), \ldots, h(n))$.  
Although $0$ is not technically in the domain of $h$, it will be convenient for us to assume $h(0)=0$.

There is a bijection from the set of all Hessenberg functions $h:[n]\to [n]$ to the set of Dyck paths of length $2n$, as follows. Each Hessenberg function $h$ may be uniquely identified with a lattice path from the upper left corner of an $n\times n$ grid to the lower right corner obtained by requiring that the horizontal step in column $i$ occurs $h(i)$ rows from the top. We call the resulting image \emph{diagram of $h$}. The diagrams of the functions $h_1=(2,3,3,4)$ and $h_2=(2,2, 4, 4)$ 
are displayed below.

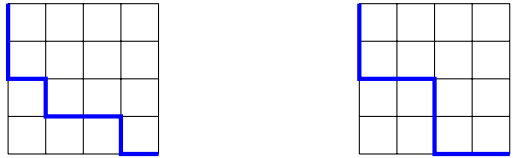
\begin{center}
\begin{tikzpicture}[scale=.5]
\draw (0,0) grid (4,4);
\draw[ultra thick,blue] 
(0,4) -- (0,2) -- (1,2) -- (1,1) -- (3,1)-- (3,0) -- (4,0);
\end{tikzpicture}
\qquad\qquad\qquad
\begin{tikzpicture}[scale=.5]
\draw (0,0) grid (4,4);
\draw[ultra thick,blue] 
(0,4) -- (0,2) -- (2,2) --  (2,0) -- (4,0);
\end{tikzpicture}
\end{center}
The diagrams help visualize examples presented in later sections.

For $j\in [n]$, let $v_j$ denote the $j$th column vector of the matrix $Z = (z_{ij})$, i.e.
\begin{equation}\label{eqn.col}
    v_j := \sum_{i=1}^n z_{ij}e_i.
\end{equation}
Let $\mx\in \mathfrak{gl}_n(\C)$. Define the ideal $\ci_{\mx,h, i}\subseteq \C[\mathbf{z}]$ as follows. If $h(i)=n$, then $\ci_{\mx,h, i}:= 0$. 
For each $i\in [n]$ such that $h(i)<n$, define $\ci_{\mx,h, i}\subseteq \C[\mathbf{z}]$ to be the ideal in $\C[\mathbf{z}]$ generated by the rank condition: 
\begin{eqnarray}\label{eqn.rank.cond}
\mathrm{rk} \begin{bmatrix} | & | &  & | & | & | &  & | \\  \mx v_1 & \mx v_2 & \cdots & \mx v_i & v_1 & v_{2} & \cdots & v_{h(i)}\\ | & | &  & | & | & | &  & |
\end{bmatrix} \leq h(i).
\end{eqnarray}
That is, $\ci_{\mx, h, i}$ is the ideal generated by all $(h(i)+1)\times (h(i)+1)$-minors of the $n\times (h(i)+i)$-matrix above.

\begin{rem}\label{rem:LiKh(i)} When $\semi = \mathrm{diag}(1,0, \dots, 0)$, the ideal $\ci_{\semi, h, i}$ is determined by the rank condition:
\begin{eqnarray*}
\mathrm{rk} \begin{bmatrix} z_{11} & z_{12} & \cdots  & z_{1i} & z_{11} & z_{12} & \cdots & z_{1h(i)} \\  0 & 0 & \cdots & 0 & z_{21} & z_{22} & \cdots & z_{2h(i)}\\  \vdots & \vdots & \ddots & \vdots & \vdots & \vdots & \ddots & \vdots\\ 0 & 0 & \cdots & 0 & z_{n1} & z_{n2} & \cdots & z_{nh(i)}
\end{bmatrix} \leq h(i).
\end{eqnarray*}
In particular, $\ci_{\semi, h, i} = \cl_i \cdot \ck_{h(i)}$, with $\cl_i$ and $\ck_j$ defined in~\eqref{eq:LiKi}.
\end{rem}

We define
\begin{eqnarray}\label{eqn.hess-ideal}
\ci_{\mx, h} : = \sum_{i\in [n]} \ci_{\mx, h,  i}\subseteq \C[\mathbf{z}].
\end{eqnarray}
Consider the action of $B$ on $M_n(\C)$ by right multiplication. This induces an action of $B$ on $\C[\mathbf{z}]$ defined on each variable $z_{ij}$ by 
$$z_{ij}\cdot b = \sum_k z_{ik}b_{kj}= \sum_{k\leq j} b_{kj}z_{ik}$$
for each upper triangular matrix $b=(b_{ij})\in B$.  This action of $B$ amounts to rescaling each column vector $v_j$ by a nonzero value, and adding any multiple of `earlier' columns $v_\ell$ with $\ell<j$. 
Note that the ideal $\ci_{\mx, h}$ is obviously invariant under this action of $B$ on $\C[\mathbf{z}]$ since the rank condition~\eqref{eqn.rank.cond} is invariant under such column operations.

Recall that $G=GL_n(\C)$ is a principal open subset of the affine scheme $M_n(\C) = \Spec (\C[\mathbf{z}])$ of $n\times n$ matrices.  As is typical, we identify $GL_n(\C)$ as an affine subscheme of $\Spec(\C[\mathbf{z}, y])$ with coordinate ring $\C[\mathbf{z}, d^{-1}]:= \C[\mathbf{z}, y]/ \left< yd-1\right>$ obtained via localization, where $d\in \C[\mathbf{z}]$ denotes the determinant function $\det(Z)$. Let  
\begin{eqnarray}\label{eqn.GLn.coord}
\iota: \C[\mathbf{z}]\hookrightarrow  \C[\mathbf{z}, d^{-1}]
\end{eqnarray}
denote the canonical injection.  Given an ideal $\ci\subseteq \C[\mathbf{z}]$, we write $\tilde{\ci}$ to denote the preimage of the ideal generated by $\iota(\ci)$.

\begin{definition} Given a matrix $\mx\in \mathfrak{gl}_n(\C)$ and Hessenberg function $h:[n]\to [n]$, the associated \emph{matrix Hessenberg scheme} is the affine scheme $$\cy_{\mx, h}:= \Spec(\C[\mathbf{z}]/\tilde{\ci}_{\mx, h}).$$
\end{definition}

Recall that $\pi: G\to G/B$ defines a correspondence between (right) $B$-invariant subschemes of $G$ and subschemes of $G/B$.
Thus, given a Hessenberg scheme $Y_{\mx, h}$
in the flag variety, we may instead consider the (right) $B$-invariant subscheme $\pi^{-1}(Y_{\mx,h})$ of $GL_n(\C)$.  By construction, the matrix Hessenberg scheme $\cy_{\mx, h}$ is precisely the Zariski closure of $\pi^{-1}(Y_{\mx, h})$ in $M_n(\C)=\Spec(\C[\mathbf{z}])$.   Furthermore, though the inclusion 
$\ci_{\mx, h} \subset \tilde{\ci}_{\mx, h}$ 
may be strict, $\cy_{\mx, h}$ is precisely the closure of the open set of prime ideals in the scheme $\Spec (\C[\mathbf{z}] / \ci_{\mx, h})$ that do not contain the determinant function $d$. 
We summarize these observations in the following remark.

\begin{rem}\label{rem.quotient} Given  the matrix $\mx\in \mathfrak{gl}_n(\C)$ and Hessenberg function $h: [n] \to [n]$, let  $Y_{\mx, h}$ denote corresponding the Hessenberg scheme in the flag variety $GL_n(\C)/B$ (also refered to as the Hessenberg variety).  Then
\begin{enumerate}
\item $\pi^{-1}(Y_{\mx, h}) = GL_n(\C)\cap \Spec(\C[\mathbf{z}] / \ci_{\mx, h})$, and
\vspace*{.08in}
\item $\cy_{\mx, h} = \overline{\pi^{-1}(Y_{\mx, h})}$.  
\end{enumerate}
\end{rem}

A natural question is to identify suitable generators of $\tilde{\ci}_{\mx, h}$.  Here ``suitable'' means explicit generators that can be used for computation using commutative algebra, such as a Gr\"obner basis.  In general, the ideals $\ci_{\mx, h}$ and $\tilde{\ci}_{\mx, h}$ may be distinct, demonstrated by Example~\ref{ex.notequal}. 
However, when the ideals {\em are} equal, we can leverage tools from commutative algebra on $\C[\mathbf{z}]$ to study matrix Hessenberg schemes via the generators of $\ci_{\mx, h}$.  This is exactly the approach of our arguments  below for $\mx$ a minimal semisimple element.

While the equations defining $\ci_{\mx, h}$ are given by rank conditions, it isn't obvious how to write down equations defining $\tilde{\ci}_{\mx, h}$.
For example, Berget and Fink~\cite{Berget-Fink2018} study ideals defining certain torus orbits in the set of $r\times n$ matrices over $\C$, and find that even in this case
the answer is not straightforward.

\begin{example}\label{ex.notequal} Let $n=5$ and $h=(3,3,5,5,5)$.  Let $E_{ij}$ denote the elementary matrix with $1$ in the $(i,j)$-th entry and all other entries equal to $0$.  Let $\mx= E_{12}+E_{23}+E_{34}+E_{45}$; in this case $Y_{\mx, h}$ is called a regular nilpotent Hessenberg variety, and $Y_{\mx, h}$ is known to be irreducible and reduced (see~\cite{ADGH2018, Abe-Fujita-Zeng2020}).  It is easy to confirm that $\ci_{\mx, h} = \ci_{\mx, h, 2}$ (see Lemma~\ref{lem.corners} below).  Thus,  $\ci_{\mx, h}$ is generated by all $4\times 4$ minors of 
\[
\begin{bmatrix} | & | &  | & | & |  \\  \mx v_1 & \mx v_2 & v_1 & v_2 & v_3 \\ | & | & | & | & | 
\end{bmatrix} = \begin{bmatrix} z_{21} & z_{22} & z_{11} & z_{12} & z_{13}\\ z_{31} & z_{32} & z_{21} & z_{22} & z_{23}\\ z_{41} & z_{42} & z_{31} & z_{32} & z_{33}\\ z_{51} & z_{52} & z_{41} & z_{42} & z_{43}\\ 0 & 0 & z_{51} & z_{52} & z_{53}  \end{bmatrix}.
\]
Since $Y_{\mx, h}$ is irreducible, the matrix Hessenberg scheme $\cy_{\mx, h}$ is also irreducible.  However, Macaulay2 can be used to verify that $\ci_{\mx, h}$ is not prime, but the intersection of two prime ideals. One of these ideals defines an affine subscheme that does not intersect $GL_5(\C)$.  In particular, $\ci_{\mx, h}\subsetneq \tilde{\ci}_{\mx, h}$.
\end{example}

\begin{rem}\label{rem.isom} Let $h$ be a fixed Hessenberg function. Matrix Hessenberg schemes for $h$ corresponding to conjugate matrices are isomorphic. Indeed, suppose $\mx, \mx' \in \mathfrak{gl}_n(\C)$ such that 
$\mx = g\mx'g^{-1}$ for some $g = (g_{ij})_{1\leq i,j\leq n}\in G$. The action of $g$ on the matrix $Z \in M_n(\C[\mathbf{z}])$ by left multiplication induces an isomorphism $\C[\mathbf{z}] \to \C[\mathbf{z}]$ defined by $z_{ij} \mapsto  g_{i1}z_{1j}+g_{i2}z_{2j} +\cdots + g_{in}z_{nj}$ for each $1\leq i, j\leq n$. This isomorphism maps $\ci_{\mx, h}$ onto $\ci_{\mx',h}$ and thus results in 
an isomorphism of schemes.
\end{rem}

The definition of the ideal $\ci_{\mx, h}$ given in~\eqref{eqn.hess-ideal} is overdetermined by the sum. It is sufficient to restrict to  `corners' of the Hessenberg function $h$.

\begin{definition} Let $h:[n]\to [n]$ be a Hessenberg function. We say $i\in[n]$ is a \emph{corner of $h$} whenever $h(i)>h(i-1)$.  Denote the set of all corners by $\cc(h)$.
\end{definition}

Our convention that $h(0)=0$ implies $1$ is always a corner of $h$. The corners of $h$ are an indexing set for the columns in the diagram of $h$ containing corners.  More specifically, $i\in \cc(h)$ implies that $(h(i), i)$ is a corner box in the diagram of $h$. 
For each $i\in [n]$, let
\[
i^*:= \max\{ k\in [n] \mid h(k)=h(i) \}.
\]
Notice that if $i$ is a corner of $h$, then $i^*+1\in \cc(h)$. Furthermore, there is no corner $k$ with $i<k<i^*+1$ in~$\cc(h)$.

\begin{example} The Hessenberg function $h=(4,4,4,6,6,6)$ has corners $1$ and $4$ with $1^*= 3$ and $4^*=6$. The function $h=(3,3,4,5,6,6)$ has corners $1$, $3$, $4$, and $5$ with $1^*=2$, $3^*=3$, $4^*=4$, and $5^*=6$.  The diagrams for these Hessenberg functions are displayed below, with corner boxes $(h(i),i)$ highlighted. 
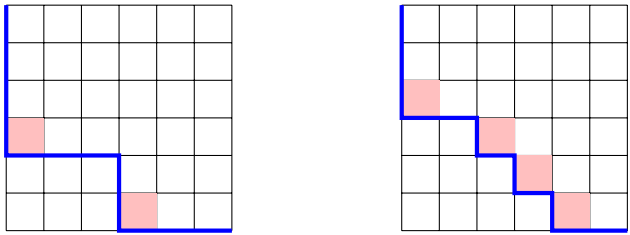
\begin{center}
\begin{tikzpicture}[scale=.5]
\draw (0,0) grid (6,6);
\draw[fill=pink, draw=none] (0,2) rectangle (1,3);
\draw[fill=pink, draw=none] (3,0) rectangle (4,1);
\draw[ultra thick,blue] 
(0,6) -- (0,2) -- (3,2) -- (3,0) -- (6,0);
\end{tikzpicture}
\quad\quad\quad\quad\quad
\begin{tikzpicture}[scale=.5]
\draw (0,0) grid (6,6);
\draw[fill=pink, draw=none] (0,3) rectangle (1,4);
\draw[fill=pink, draw=none] (2,2) rectangle (3,3);
\draw[fill=pink, draw=none] (3,1) rectangle (4,2);
\draw[fill=pink, draw=none] (4,0) rectangle (5,1);
\draw[ultra thick,blue] 
(0,6) -- (0,3) -- (2,3) -- (2,2) -- (3,2)-- (3,1) -- (4,1) -- (4,0) -- (6,0);
\end{tikzpicture}
\end{center}
\end{example}

\begin{lemma}\label{lem.corners} For all Hessenberg functions $h$ and matrices $\mx \in \mathfrak{gl}_n(\C)$,
\[
\ci_{\mx, h} = \sum_{i\in \cc(h)} \ci_{\mx, h, i ^*}.
\]
Recall that $\ci_{\mx, h, n}=0$, so the corner $i\in \cc(h)$ such that $h(i)=n$ does not contribute to the sum. 
\end{lemma}
\begin{proof} The desired formula follows from that fact that if $i,j\in [n]$ such that $i<j$ and $h(i)=h(j)$, then $\ci_{\mx,h, i} \subseteq \ci_{\mx,h, j}$.  Indeed, our assumptions imply that the matrix
\begin{eqnarray}\label{eqn.1}
\begin{bmatrix} | & | &  & | & | & | &  & | \\  \mx v_1 & \mx v_2 & \cdots & \mx v_i & v_1 & v_{2} & \cdots & v_{h(i)}\\ | & | &  & | & | & | &  & |
\end{bmatrix} 
\end{eqnarray}
can be obtained from 
\begin{eqnarray}\label{eqn.2}
\begin{bmatrix} | & | &  & | & | & | &  & | \\  \mx v_1 & \mx v_2 & \cdots & \mx v_j & v_1 & v_{2} & \cdots & v_{h(j)}\\ | & | &  & | & | & | &  & |
\end{bmatrix}
\end{eqnarray}
by deleting the columns $\mx v_{i+1}, \ldots, \mx v_{j}$.  Furthermore, since $h(i)=h(j)$, the collection of all $(h(i)+1)\times (h(i)+1)$-minors of~\eqref{eqn.1} is a subset of all $(h(j)+1)\times (h(j)+1)$-minors of~\eqref{eqn.2}. Thus $\ci_{\mx,h, i} \subseteq \ci_{\mx,h, j}$ as desired.
\end{proof}

\begin{example}\label{ex.2444.1}
Let $n=4$ and $h=(2,4,4,4)$. The corner set of $h$ is $\cc(h)=\{1,2\}$, and we have   $1^*=1$ and $2^*=4$. The diagram for $h$ is displayed below, with the corner boxes $(h(i),i)$ highlighted. 
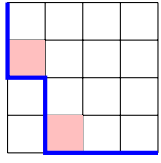
\begin{center}
\begin{tikzpicture}[scale=.5]
\draw (0,0) grid (4,4);
\draw[fill=pink, draw=none] (0,2) rectangle (1,3);
\draw[fill=pink, draw=none] (1,0) rectangle (2,1);
\draw[ultra thick,blue] 
(0,4) -- (0,2) -- (1,2) -- (1,0)  -- (4,0);
\end{tikzpicture}
\end{center}
Since $  \ci_{\mx, h, 4}=0$,  Lemma~\ref{lem.corners} implies 
$\ci_{\mx, h} = \ci_{\mx, h, 1},
$
which is generated by the $3\times 3$ determinantal conditions derived from
\begin{eqnarray}\label{rank.h2444}
\rk\begin{bmatrix} | & | & | \\  \mx v_1 &  v_1 & v_2 \\ | & | &  |  \end{bmatrix}\leq 2.
\end{eqnarray}
\end{example}

\begin{example}\label{ex.2444.3}
We continue Example~\ref{ex.2444.1}, with $h= (2, 4,4,4)$ and $\mx=\semi = \mathrm{\diag}(1,0,0,0)$. The rank conditions of \eqref{rank.h2444} are
$$
 \rk\begin{bmatrix} 
z_{11} &  z_{11} &  z_{12}\\ 
0	&  z_{21} &  z_{22}\\
0&  z_{31} &  z_{32}\\
0&  z_{41} &  z_{42}\\
 \end{bmatrix} \leq 2.
 $$
The rank conditions are satisfied exactly when all  $3\times 3$ minors vanish,  so
 $$
 \ci_{\semi, h} = \langle z_{11}(z_{21}z_{32}-z_{22}z_{31}), z_{11}(z_{21}z_{42}-z_{22}z_{41}), z_{11}(z_{31}z_{42}-z_{32}z_{41})
 \rangle
 $$
Macaulay2 confirms that $\ci_{\semi, h}$ is a reduced ideal, with associated primes
\[
\ck_2 = \langle z_{32}z_{41}-z_{31}z_{42},z_{22}z_{41}-z_{21}z_{42},z_{22}z_{31}-z_{21}z_{32}\rangle \; \mbox{and} \;
\cl_1 =  \langle z_{11}\rangle.
\]
We identify these as Schubert determinantal ideals using Lemma~\ref{lem.Schubert}. Thus $\cy_{\semi, h}$ is a union of the corresponding matrix Schubert varieties:
$$
\cy_{\semi, h}= \cx_{w_0[1423]} \cup \cx^{\op}_{[2134]} =  \cx_{[4132]} \cup \cx^{\op}_{[2134]}.
$$
\end{example}


\section{The main theorem}\label{sec.proofs}

Suppose $\mx\in \mathfrak{gl}_n(\C)$ is a minimal semisimple element.  It is straightforward to show that any such matrix is conjugate to $\semi'= \diag(c_1,c_2,\ldots,c_2)$ for some $c_1,c_2\in \C$ such that $c_1\neq c_2$. The rank inequalities of \eqref{eqn.rank.cond} are satisfied for $\semi'$ if and only if they are satisfied for $\semi'' = \diag(c_1-c_2,0, \dots, 0)$ which are themselves satisfied if and only if they are satisfied for $\semi = \diag(1, 0,\dots, 0)$. Thus $\ci_{\semi', h} = \ci_{\semi ,h}$ for any fixed choice of Hessenberg function $h$. Using Remark~\ref{rem.isom}, we conclude that any minimal semisimple Hessenberg scheme $\cy_{\mx, h}$ is isomorphic to $\cy_{\semi, h}$.  
For this reason, it is sufficient to restrict our attention to the  
geometry
of $\cy_{\semi, h}$.

\begin{thm}\label{thm.ss.components}  Let $\semi = \mathrm{\diag} (1,0,\ldots, 0)$ and let $h:[n]\to[n]$ be a Hessenberg function. The matrix Hessenberg scheme  $\cy_{\semi, h}$ is a reduced union of matrix Richardson schemes:
\[
\cy_{\semi, h} = \bigcup_{i\in \cc(h)} \left(\cx_{u[i]}^{\op} \cap \cx_{w_0v[h(i)]} \right).
\]
Moreover, the ideal defining $\cy_{\semi, h}$ satisfies 
\begin{equation}\label{eqn.intersection}
\tilde{\ci}_{\semi, h} = \ci_{\semi, h}  =  \bigcap_{i\in \cc(h)} \left( \cl_{i-1} + \ck_{h(i)}  \right).
\end{equation}
\end{thm}

\begin{proof} The statement about the structure of the scheme follows from the claim about the ideals by Lemma~\ref{lem.Schubert4}.   Recall the definition of $\cl_i$ and $\ck_j$ in Equation~\ref{eq:LiKi}.
By Remark~\ref{rem:LiKh(i)}, $\ci_{\semi, h, i} = \cl_i \cdot \ck_{h(i)}$. Then
\begin{equation}\label{eqn.sum}
\ci_{\semi, h} = \sum_{i\in \cc(h)}\ci_{\semi, h, i}=\sum_{i \in \cc(h)} \cl_{i^*} \cdot \ck_{h(i)} =\sum_{i\in \cc(h)} \cl_{i^*} \cap \ck_{h(i)}
\end{equation}
where the first equality follows from Lemma~\ref{lem.corners} and the last equality follows from Lemma~\ref{lem.intersection}. We aim to prove the second equality of~\eqref{eqn.intersection} using induction on the size of $\cc(h)$. 

If $h$ has one corner, then $h=(n,n,\dots, n)$. It follows that $1^*=h(1)=n$, and~\eqref{eqn.sum} becomes
$$
\ci_{\semi,h} = \cl_{1^*} \cap \ck_{h(1)} = \cl_{n} \cap \ck_{n} = \langle 0\rangle
$$
since $\ck_n =  \langle 0\rangle$. On the other hand, $\bigcap_{i\in \cc(h)} \left(\cl_{i-1}+\ck_{h(i)}\right)= \cl_0+\ck_n = 0$, since $\ck_n =\cl_0= \langle 0\rangle$. This proves the second equality of~\eqref{eqn.intersection} in this case.

Now suppose the equality holds for any Hessenberg function whose corner set has size $k$, and suppose $|\cc(h)|=k+1$.  Observe that  $\{i\in \cc(h) \mid h(i)<n\}$ is nonempty, since $\cc(h)$ has at least two elements. Let $m=\max\{i\in \cc(h) \mid h(i)<n\}$.  Then $m^*+1$ is the maximum element of $\cc(h)$, and  $h(m^*+1)=n$.

Define $h':[n]\to [n]$ to be the Hessenberg function such that $h'(i)=h(i)$ for all $i<m$ and $h'(i)=n$ for all $i\geq m$. By construction, $\cc(h') = \cc(h) \setminus \{m^*+1\}$,  $m$ is the largest corner of $h'$, and $\ck_{h(i)} = \ck_{h'(i)}$ for all $i<m$. 
Applying~\eqref{eqn.sum} and the inductive assumption to $h'$ and the fact that $\ck_n= \langle 0\rangle$, we obtain
\begin{align*}
\sum_{i\in \cc(h),\, i<m} \cl_{i^*}\cap \ck_{h(i)} &= \sum_{i\in \cc(h')} \cl_{i^*}\cap \ck_{h'(i)}=\ci_{\semi, h'} = \bigcap_{i\in \cc(h')} \left(\cl_{i-1}+ \ck_{h'(i)}\right)\\
& = \left( \bigcap_{i\in \cc(h),i< m}\left(\cl_{i-1}+ \ck_{h(i)}\right) \right) \cap \cl_{m-1}.
\end{align*}
Let $\cj = \bigcap\limits_{i\in \cc(h),\, i<m} \left( \cl_{i-1}+\ck_{h(i)} \right)$ for notational simplicity. 
By the previous calculation,
 \begin{align*}
\ci_{\semi, h} =\left( \sum_{i\in \cc(h), i< m} \cl_{i^*}\cap \ck_{h(i)}\right) + \cl_{m^*}\cap \ck_{h(m)}   = \cj\cap \cl_{m-1} + \cl_{m^*}\cap \ck_{h(m)} 
 \end{align*}
 where the first equality follows from~\eqref{eqn.sum} and $\ck_{h(m^*+1)}=\ck_n=\langle 0 \rangle$. Note that $h(i)\leq h(m)$ for all $i<m$, so that $\ck_{h(m)} \subseteq \cl_{i-1}+\ck_{h(i)}$ and thus $\ck_{h(m)} \subseteq \cj$. Therefore, 
\begin{align*}
\ci_{\semi, h}& =\left( \cj \cap  \cl_{m-1} \right) + (\cl_{m^*}\cap \ck_{h(m)}) \\
&=\left(  \cl_{m^*} \cap \cj \cap  \cl_{m-1} \right) + (\cl_{m^*}\cap \ck_{h(m)}) && \text{ since $\cl_{m-1}\subseteq \cl_{m^*},$ as $m-1<m\leq m^*$}\\
&=\cl_{m^*} \cap \left[ \left( \cj \cap  \cl_{m-1} \right) +\ck_{h(m)} \right] &&  \textup{ 
since $\cj\cap \cl_{m-1}\subseteq \cl_{m^*}$} \\
&=\cl_{m^*} \cap \left[ \left( \cj \cap  \cl_{m-1} \right) + (\cj \cap \ck_{h(m)}) \right] && \textup{ since $\ck_{h(m)} \subseteq \cj$}\\
&= \cj \cap  \left( \cl_{m-1} + \ck_{h(m)} \right)  \cap \cl_{m^*} \\
&=\bigcap_{i\in \cc(h)} \left(\cl_{i-1} +\ck_{h(i)}\right)
\end{align*}
where the last equality follows from the definition of $\cj$ and the fact that $\cl_{m^*} +\ck_{h(m^*+1)} = \cl_{m^*} + \ck_n =  \cl_{m^*}$. This proves the second equality of~\eqref{eqn.intersection}. 

Each ideal $\cl_{i-1}+\ck_{h(i)}$ in~\eqref{eqn.intersection} is prime by Lemma~\ref{lem.Schubert4}, so $\ci_{\semi, h}$ is an intersection of prime ideals and therefore radical.  By Lemma~\ref{lemma.inclusions}, we have $\cl_{i-1}+\ck_{h(i)} \subseteq \cl_{j-1}+\ck_{h(j)}$ if and only if $i\leq j$ and $h(i)\geq h(j)$.  However, for distinct corners $i, j\in \cc(h)$, the inequality $i\leq j$ implies $h(i)<h(j)$. It follows that the primes in the set $\left\{\cl_{i-1}+\ck_{h(i)} \mid i\in \cc(h)\right\}$ are distinct and none is contained in any of the others.  Consequently, $\left\{\cl_{i-1}+\ck_{h(i)} \mid i\in \cc(h)\right\}$ is the set of associated primes of $\ci_{\semi, h}$. 

Recall that $d\in \C[\mathbf{z}]$ denotes the determinant function of the generic $n\times n$ matrix $Z=(z_{ij})_{1\leq i,j \leq n}$.  If $d\in \cl_{i-1}+\ck_{h(i)}$ for some $i$ then $\GL_n(\C) \cap  \Spec(\C[\mathbf{z}]/(\ci_{i-1} + \ck_{h(i)})) = \emptyset$, contradicting Lemma~\ref{lem.Schubert4}.  This proves $d\notin \cl_{i-1}+\ck_{h(i)}$ for all $i$, so $$\widetilde{\cl_{i-1}+\ck_{h(i)}}=\cl_{i-1}+\ck_{h(i)}.$$
Since the localization map is injective, it preserves intersections of ideals, and the previous observation implies $\tilde{\ci}_{\semi, h} =  \ci_{\semi, h}$, proving the first equality of \eqref{eqn.intersection}. 
\end{proof}

\begin{cor}\label{cor.ss.variety} For any Hessenberg function $h:[n]\to [n]$, the minimal semisimple Hessenberg variety $Y_{\semi, h}$ is isomorphic to a reduced union of Richardson varieties,
\[
Y_{\semi, h} = \bigcup_{i\in \cc(h)} \left( X_{u[i]}^{\op} \cap X_{w_0v[h(i)]} \right).
\]
\end{cor}

Theorem~\ref{thm.ss.components} shows that the structure of the semisimple Hessenberg variety $Y_{\semi, h}$ can be read directly from the data of the Hessenberg function.  We recover the results of Example~\ref{ex.2444.3}; another example appears below. 

\begin{example} \label{ex.semisimple} Let $\semi= \mathrm{diag}(1, 0,0,0)$.  First, we consider  $h=(1,2,3,4)$.  The corner set of $h$ is $\cc(h) = \{1,2,3,4\}$ and the corresponding diagram is the following.
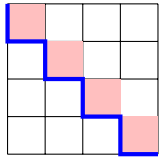
\begin{center}
\begin{tikzpicture}[scale=.5]
\draw (0,0) grid (4,4);
\draw[fill=pink, draw=none] (0,3) rectangle (1,4);
\draw[fill=pink, draw=none] (1,2) rectangle (2,3);
\draw[fill=pink, draw=none] (2,1) rectangle (3,2);
\draw[fill=pink, draw=none] (3,0) rectangle (4,1);
\draw[ultra thick,blue] 
(0,4) -- (0,3) -- (1,3) -- (1,2) -- (2,2) -- (2,1) -- (3,1)  -- (3,0) -- (4,0);
\end{tikzpicture}
\end{center}
In this case, the variety $Y_{\semi, (1,2,3,4)} = \pi(GL_4(\C) \cap \cy_{\semi, (1,2,3,4)})$ is the Grothedieck--Springer fiber over $\semi$.  We have
\[
u[1]=e,\; u[2]=[2134],\; u[3] = [2314],\; u[4] = [2341],
\]
and
\[
\; v[1] = [4123],\; v[2] = [1423],\; v[3] = [1243],\; v[4]=e.
\]
Thus, by Theorem~\ref{thm.ss.components}, $\cy_{\semi, (1,2,3,4)}$ is the union of four irreducible components
\begin{eqnarray*}
\cy_{\semi, (1,2,3,4)} &=& \cx_{w_0 v[1]} \cup \left( \cx_{u[2]}^{\op} \cap \cx_{w_0v[2]} \right) \cup \left( \cx_{u[3]}^{\op} \cap \cx_{w_0v[3]} \right) \cup \cx_{u[4]}^{\op} \\
&=& \cx_{[1432]} \cup \left( \cx_{[2134]}^{\op} \cap \cx_{[4132]} \right) \cup \left( \cx_{[2314]}^{\op} \cap \cx_{[4312]} \right) \cup \cx_{[2341]}^{\op}, 
\end{eqnarray*}
each of which is a matrix Richardson variety.

Next, consider $h=(2,2,4,4)$.   The corner set of $h$ is $\cc(h) = \{1,3\}$ and the corresponding diagram is the following.
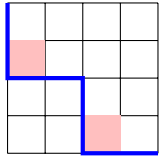
\begin{center}
\begin{tikzpicture}[scale=.5]
\draw (0,0) grid (4,4);
\draw[fill=pink, draw=none] (0,2) rectangle (1,3);
\draw[fill=pink, draw=none] (2,0) rectangle (3,1);
\draw[ultra thick,blue] 
(0,4) -- (0,2) -- (2,2) -- (2,0)  -- (4,0);
\end{tikzpicture}
\end{center}
In this case Theorem~\ref{thm.ss.components} tells us $\cy_{\semi, (2,2,4,4)}$ is the union of two irreducible components,
\[
\cy_{\semi, (2,2,4,4)} = \cx_{w_0 v[2]}  \cup \cx_{u[3]}^{\op} =  \cx_{[4132]} \cup  \cx_{[2314]}^{\op}.
\]
\end{example}

\

\begin{cor}\label{cor.dimension} Let $h:[n]\to [n]$ be a Hessenberg function. The dimension of the minimal semisimple matrix Hessenberg scheme $\cy_{\semi, h}$ is 
\[
\dim \cy_{\semi, h} = \frac{n(n+1)}{2} + \frac{(n-1)(n-2)}{2} + \max_{i\in \cc(h)} \{ h(i)-i \}.
\]
\end{cor}
\begin{proof} By Proposition~\ref{prop.Schubert}(3), the dimension of each Richardson variety $X_{u[i]}^{\op} \cap X_{w_0v[h(i)]}$ is
\begin{eqnarray*}
\ell(w_0v[h(i)]) - \ell(u[i]) &=& \left(\frac{n(n-1)}{2}-(n-h(i))\right) -(i-1)\\
&=&   \frac{(n-1)(n-2)}{2} + (h(i)-i).
\end{eqnarray*}
Since 
\begin{eqnarray}\label{eqn.dim.component}
\dim \cx_{u[i]}^{\op} \cap \cx_{w_0v[h(i)]} &=&  \dim B + \dim \left( X_{u[i]}^{\op} \cap X_{w_0v[h(i)]}\right)\\
\nonumber&=& \frac{n(n+1)}{2} + \frac{(n-1)(n-2)}{2} + (h(i)-i)
\end{eqnarray} 
the result now follows immediately from Theorem~\ref{thm.ss.components}. 
\end{proof}

It is now straightforward to determine when the minimal semisimple Hessenberg scheme is equidimensional.

\begin{cor}\label{cor.equidim}  Let $h:[n]\to [n]$ be a Hessenberg function. The minimal semisimple matrix Hessenberg scheme $\cy_{\semi, h}$ is equidimensional if and only if all of the corners of the associated diagram of $h$ lie on the same subdiagonal.  In other words, $\cy_{\semi, h}$ is equidimensional if and only if $h(i)-i=h(j)-j$ for all distinct corners $i$ and $j$ of $h$. 
\end{cor}
\begin{proof} Given distinct corners $i$ and $j$ of the Hessenberg function $h$, the formula~\eqref{eqn.dim.component} from the proof of Corollary~\ref{cor.dimension} implies that the corresponding irreducible components of $\cy_{\semi, h}$ have the same dimension if and only if $h(i)-i = h(j)-j$.  Since the subdiagonal containing the corner box $(h(i),i)$ is the set of all $(a,b)$ such that $a-b = h(i)-i$, the result is proved.
\end{proof} 

Notice that minimal semisimple Hessenberg schemes appearing in Example~\ref{ex.semisimple} are equidimensional, which is obvious from Corollary~\ref{cor.equidim} by looking at the diagrams for the corresponding Hessenberg functions.  Similarly, Example~\ref{ex.2444.3} considers a case that is not equidimensional, which is again obvious in light of Corollary~\ref{cor.equidim} since the two corners appear in different subdiagonals of the diagram for $h$.

Next, we compute the cohomology class of the minimal semisimple Hessenberg variety $Y_{\semi,h}$.  Recall Borel's description of the integral cohomology ring $H^*(GL_n(\C)/B,\Z)$ as the ring of coinvariants, that is, 
\[
H^*(GL_n/B,\Z) \simeq \Z[ x_1,\dots, x_{n}]/I, 
\]
where $I \subset \Z[ x_1,\dots, x_{n}]$ is the ideal generated by the symmetric polynomials without a constant term.  The Schubert polynomial $\mathfrak{S}_w\in \Z[ x_1,\dots, x_{n}]$ is a polynomial representative for the class $[X_w^{\op}]$ of the opposite Schubert variety $X_w^{\op}$. It is well-known that the product of Schubert polynomials $\mathfrak{S}_u \cdot \mathfrak{S}_v$ is a polynomial representative for the cohomology class $[X_u^{\op}\cap X_{w_0v}]$ of the Richardson variety $X_u^{\op}\cap X_{w_0v}$. For a more detailed definition of Schubert polynomials see~\cite{Manivel}.

With this notation in place, Corollary~\ref{cor.ss.variety} allows us to compute a representative for $[Y_{\semi, h}]$ as a sum of Schubert polynomials.

\begin{cor} \label{cor.cohomclass} Suppose $n\geq 3$, and let $h$ be a Hessenberg function and $d_h= \max \{h(i)-i \mid i\in [n]\}$. The cohomology class of the Hessenberg variety $Y_{\semi, h}$ is 
$$
[Y_{\semi, h}]= 
\begin{cases}
\displaystyle{\ \sum_{i=1}^{n-1}\ 2\ \mathfrak{S}_{w[i+1,i]}} & \mbox{if $h=(1,2,\dots, n)$}\\
\\
 \displaystyle{\sum_{\substack{i\in \cc(h)\\ d_h = h(i)-i}} \mathfrak{S}_{w[ i, h(i)]} }& \mbox{if $h\neq (1,2,\dots, n)$.}
\end{cases}
$$
where $w[i,j]$ is the shortest permutation $w$ in $S_n$ with $w(i)=1$ and $w(j) = n$,
\end{cor}
\begin{proof} 
By Corollary~\ref{cor.ss.variety}, 
$$
Y_{\semi, h} = \bigcup_{i\in \cc(h)}  \left(X^{\op}_{u[i]} \cap X_{w_0v[h(i)]} \right).
$$
The largest dimensional terms occur when $h(i)-i=d_h$. The cohomology class is therefore given by 
\begin{align*}
[Y_{\semi, h}]&= \displaystyle{\sum_{\substack{i\in \cc(h)\\ d_h = h(i)-i}} \mathfrak{S}_{u[i]}\cdot \mathfrak{S}_{v[h(i)]}}.
\end{align*}
Consider first the case that $h\neq (1, 2,\dots, n)$. Then $d_h>0$ and thus $i<h(i)$ for each $i$ in the sum. Since $i<h(i)$, $u[i]$ and $v[h(i)]$ are in commuting subgroups of the permutation group, implying $\mathfrak{S}_{u[i]}\mathfrak{S}_{v[h(i)]} = \mathfrak{S}_{u[i]v[h(i)]}=\mathfrak{S}_{w[i,h(i)]}.$ The second equality follows.  

When $h=(1, 2, \dots, n)$, observe that $\cc(h)=[n]$ and $d_h=0$ is achieved for all $i\in [n]$. Since $u[1]=v[n]=e$, the expression simplifies to
\begin{align*}\label{eq:Piericase}
[Y_{\semi, h}]&= \displaystyle{\sum_{i=1}^{n} \mathfrak{S}_{u[i]}\cdot \mathfrak{S}_{v[i]}} =  \mathfrak{S}_{v[1]}+ \left( \displaystyle{\sum_{i=2}^{n-1} \mathfrak{S}_{u[i]}\cdot \mathfrak{S}_{v[i]}}\right) + \mathfrak{S}_{u[n]}.
\end{align*}
The terms in the middle sum can be evaluated using the fact that $u[i]=s_1s_2\cdots s_{i-1}$ (here $s_i$ denotes the simple transposition swaping $i$ and $i+1$) and the Pieri rule, which dictates that 
$$
\mathfrak{S}_{u[i]}\cdot \mathfrak{S}_{v[i]}= \mathfrak{S}_{w[i, i-1]}+ \mathfrak{S}_{w[i+1, i]}, \quad \mbox{for }1<i<n.
$$

Since $v[1]=w[2,1]$ and $u[n]=w[n,n-1]$, we obtain 
\begin{align*}
[Y_{\semi, h}]&= \mathfrak{S}_{w[2,1]}+\left(\displaystyle{\sum_{i=2}^{n-1} 
\left(\mathfrak{S}_{w[i, i-1]}+ \mathfrak{S}_{w[i+1, i]}\right)
} \right) + \mathfrak{S}_{w[n,n-1]} = \sum_{i=1}^{n-1} \ 2\ \mathfrak{S}_{w[i+1, i]}.
\end{align*}
This concludes the proof.
\end{proof}
We show in \cite{GP-families} that $\cy_{\semi,h}$ admits a flat degeneration to the minimal nilpotent matrix Hessenberg scheme $\cy_{\nilp, h}$. Since the underlying variety of $\cy_{\nilp, h}$ is a union of matrix Schubert varieties, the factor of 2 occurring when $h=(1,2,\dots, n)$ implies that $\cy_{\nilp, h}$ is not a reduced scheme.

\end{document}